\documentclass[leqno,12pt]{amsart}

\usepackage{tikz}
 \usetikzlibrary{calc}
\usepackage{amsmath, amssymb}
\usepackage{amsthm} 
\theoremstyle{plain} 
\newtheorem{theorem}{\indent\sc Theorem}[section]
\newtheorem{lemma}[theorem]{\indent\sc Lemma}
\newtheorem{corollary}[theorem]{\indent\sc Corollary}

\newtheorem{conjecture}[theorem]{\indent\sc Conjecture}

\theoremstyle{definition} 

\newcommand{\Z}{\mathbb{Z}}
\newcommand{\R}{\mathbb{R}}

\newcommand{\B}{\mathcal{B}}

\newcommand{\cone}{\mathrm{cone}}

\def\rddots{\cdot^{\cdot^{\cdot}}}

\begin{document}

\title{Strong factorization and the braid arrangement fan}

\author[J. Machacek]{John Machacek$^*$}


\thanks{ 
$^{*}$With support of the York Science Fellowship.
}

\address{
Department of Mathematics and Statistics\endgraf
York University\endgraf
Toronto, Ontario M3J 1P3\endgraf
Canada
}
\email{machacek@york.ca}
\maketitle

\begin{abstract}
We establish strong factorization for pairs of smooth fans which are refined by the braid arrangement fan.
Our method uses a correspondence between cones and preposets.
\end{abstract}

\section{Introduction}
We use the combinatorics of the braid arrangement fan to study toric varieties coming from fans refined by the braid arrangement fan.
The braid arrangement fan is the fan whose maximal cones are the Weyl chambers of type $A$.
Our main result is a proof of Oda's strong factorization conjecture in the special case where we are given two smooth fans that are refined by the braid arrangement fan.
We now recall Oda's strong factorization conjecture is its combinatorial form.

\begin{conjecture}[\cite{Oda}]
Given two smooth fans $\Sigma_1$ and $\Sigma_2$ with the same support, there exists a third fan $\Sigma_3$ which can be obtain from both $\Sigma_1$ and $\Sigma_2$ by sequences of smooth star subdivisions.
\label{conj:Oda}
\end{conjecture}

Given a birational map we can consider the problem of finding a \emph{factorization} of the map.
A \emph{strong factorization} consists of a sequence of blow-ups followed by a sequence of blow-downs.
A \emph{weak factorization} allows blow-ups and blow-downs in any order.
For a birational map between two smooth complete varieties over an algebraically closed field of characteristic zero a weak factorization exists~\cite{AKMW, Wlod2003}.
The existence of a strong factorization is an open problem.

Equivariant versions of the factorization problems for smooth toric varieties were conjectured in~\cite{Oda} and became known as ``Oda's strong factorization conjecture'' and ``Oda's weak factorization conjecture.''
As with many problems in toric geometry, Oda's conjectures can be phrased in terms of convex geometry as we do in Conjecture~\ref{conj:Oda}.
In these terms blow-ups become star subdivisions.
Oda's weak factorization conjecture has been solved by W\l odarczyk~\cite{Wlod1997} and Morelli~\cite{Morelli}.
Oda's strong factorization conjecture is still open.
Da Silva and Karu~\cite{DaSK2011} propose an algorithm to produce a strong factorization from a weak factorization, but this proposed algorithm is not guaranteed to terminate.

When refining a fan to produce a factorization we need to have some way a bounding the complexity of the fans that will arise in order to avoid the problem of the process not terminating.
In our case we find that all refinements produced will be coarsenings the braid arrangement fan.
In Section~\ref{sec:braid} we review the braid arrangement fan and its combinatorics. 
Section~\ref{sec:main} contains the proof on our main result that establishes Conjecture~\ref{conj:Oda} in the case that $\Sigma_1$ and $\Sigma_2$ are refined by the braid arrangement fan.

\section{The braid arrangement fan}
\label{sec:braid}
Let $[n] = \{1,2,\dots,n\}$ and consider the lattice $\mathbb{Z}^n \subseteq \mathbb{R}^n$ with standard basis $\{e_i : i \in [n]\}$.
For any $A \subseteq [n]$ we let $e_A = \sum_{i \in A} e_i$.
We then take the lattice $N = \mathbb{Z}^n / \mathbb{Z} e_{[n]}$ and vector space $N_{\mathbb{R}} = \mathbb{R} \otimes_{\mathbb{Z}} N$.
For any $1\leq i < j \leq n$ let $H_{ij}$ denote the hyperplane in $N_{\mathbb{R}}$ defined by $x_i = x_j$.
The \emph{braid arrangement} is the hyperplane arrangement consisting of the hyperplanes $\{H_{ij}\}_{1 \leq i < j \leq n}$.
We let $\B(n)$ denote the \emph{braid arrangement fan} in $N_{\mathbb{R}}$.
The maximal cones in $\B(n)$ are Weyl chambers of type $A_{n-1}$.
We let $\Pi(n)$ denote the \emph{permutahedron}.
This is a lattice polytope with whose normal fan is $\B(n)$.
The vertices of $\Pi(n)$ are all permutations of the vector $(1,2,\dots, n)$.
The rays in $\B(n)$ are $e_{A}$ for each $A \subseteq [n]$ such that $A \not\in \{\emptyset, [n]\}$.
The permutahedron $\Pi(3)$ and the fan $\B(3)$ with rays label by minimal lattice points are shown in Figure~\ref{fig:Pi3}.
\begin{figure}
    \centering
    \begin{tikzpicture}
    \node (0) at (0,0) {};
    \node (1) at (-3.464,-2) {};
    \node at (-3.7,-2.15) {$e_{\{1\}}$};
    \node (2) at (3.464,-2) {};
    \node at (3.7,-2.15) {$e_{\{2\}}$};
    \node (3) at (0,4) {};
    \node at (0,4.2) {$e_{\{3\}}$};
    \node (12) at (0,-4) {};
    \node at (0,-4.2) {$e_{\{1,2\}}$};
    \node (13) at (-3.464,2) {};
    \node at (-3.7,2.15) {$e_{\{1,3\}}$};
    \node (23) at (3.464,2) {};
    \node at (3.7,2.15) {$e_{\{2,3\}}$};
    \draw[-{latex}] (0,0) to (1);
    \draw[-{latex}] (0,0) to (2);
    \draw[-{latex}] (0,0) to (3);
    \draw[-{latex}] (0,0) to (12);
    \draw[-{latex}] (0,0) to (13);
    \draw[-{latex}] (0,0) to (23);
    \draw[ultra thick] (1.5,2.598) -- (3,0) -- (1.5,-2.598)-- (-1.5,-2.598) -- (-3,0) -- (-1.5,2.598) --  (1.5,2.598);

    \node[draw,circle,fill=black,scale=0.35] at (1.5,2.598) {};

    \node[rotate=-30] at (1.7,2.78) {$(1,2,3)$};
    \node[draw,circle,fill=black,scale=0.35] at (3,0) {};

    \node at (3.7,0) {$(1,3,2)$};
    \node[draw,circle,fill=black,scale=0.35] at (1.5,-2.598) {};

    \node[rotate=30] at (1.7,-2.78) {$(2,3,1)$};
    \node[draw,circle,fill=black,scale=0.35] at (-1.5,-2.598) {};

    \node[rotate=-30] at (-1.7,-2.78) {$(3,2,1)$};
    \node[draw,circle,fill=black,scale=0.35] at (-3,0) {};

    \node at (-3.7,0) {$(3,1,2)$};
    \node[draw,circle,fill=black,scale=0.35] at (-1.5,2.598) {};

    \node[rotate=30] at (-1.7,2.78) {$(2,1,3)$};
    \end{tikzpicture}
    \caption{The permutahedron $\Pi(3)$ and its normal fan $\B(3)$.}
    \label{fig:Pi3}
\end{figure}
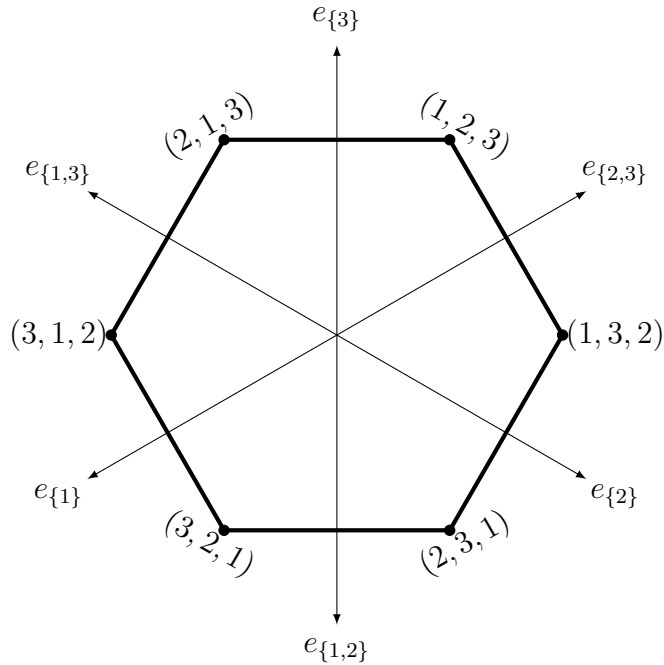

A \emph{preposet} is a binary relation which is reflexive and transitive. 
We will use either ordered pairs $(i,j)$ or $i \preceq j$ to denote relations in preposets.
Postnikov, Reiner, and Williams have given a correspondence between cones in fans which are refined by $\B(n)$ and preposts on $[n]$~\cite[Section 3]{PRW}.
When not otherwise stated we will assume that any preposet has $[n]$ as its underlying set.
The correspondence works by observing that any cone in such a fan is determine by inequalities of the form $x_i \leq x_j$ where $(x_1, x_2, \dots, x_n)$ are coordinates for $N = \Z^n/\Z e_{[n]}$.
Hence, $x_i \leq x_j$ exactly maps to $i \preceq j$.
For cones which are not of maximal dimension there will be pairs $x_i \leq x_j$ and $x_j \leq x_i$ giving equality $x_i = x_j$.
A poset is a prepost which is also antisymmetric.
A poset is called a \emph{tree poset} if its Hasse diagram is a tree.
Recall the Hasse diagram of a poset has a vertex for each element of the poset and a directed edge $i \to j$ for each covering relation $i \prec j$.

Any preposet determines an equivalence relation on $[n]$ by declaring $i \sim i$ for all $i \in [n]$ and $i \sim j$ for $i \neq j$ whenever $i \preceq j$ and $j \preceq i$.
For any preposet we define its \emph{Hasse diagram} to be the transitive reduction of the graph (with loops and multiple edges removed) which has a vertex for each equivalence class and a directed edge from the class of $a$ to the class of $b$ for each relation $a \preceq b$.
A preposet is called \emph{connected} if its Hasse diagram is connected.
We will assume throughout that all preposets we encounter are connected, this ensures that the cones they define are strongly convex~\cite[Proposition 3.5 (7)]{PRW}.
When drawing a Hasse diagram we will omit the direction of the edges and draw so that each edge should be oriented upward.
A \emph{tree preposet} is defined to be a preposet whose Hasse diagram is a tree.
Maximal cones will be labeled by posets.
In general the dimension of a cone labeled by a preposet will be one less than the number of equivalence classes on $[n]$ the preposet determines.

For a cone $\sigma$ let $\sigma(1)$ denote the set of minimal lattice points of its ray generators.
Give a set of lattice vectors $A \subseteq N$ we let
\[\cone(A) = \{\sum_{v \in A} \lambda_v v : \lambda_v \in \R, \lambda_v \geq 0\} \subseteq N_{\R}\]
and thus $\sigma = \cone(\sigma(1))$.
If $\sigma(1)$ can be extended to a basis of lattice $N$, then $\sigma$ is called \emph{smooth}.
A fan inherits the adjective smooth if each cone in the fan is smooth.
Consider a cone $\tau$ in a smooth fan $\Sigma$.
Let  $\tau(1) = \{v_1, v_2, \dots, v_k\}$ and set $v_0 = v_1 + v_2 + \dots + v_k$.
For any $\tau \subseteq \sigma \in \Sigma$ the \emph{(smooth) star subdivision} of the cone $\sigma$ relative to $\tau$ is the fan
\[\Sigma^*_{\sigma}(\tau) = \{\cone(A) : A \subseteq \sigma(1) \cup \{v_0\}, \tau(1) \not\subseteq A\}.\]
The \emph{star subdivision of the fan $\Sigma$} relative to $\tau$ is
\[\Sigma^*(\tau) = \{\sigma \in \Sigma : \tau \not\subseteq \sigma\} \cup \bigcup_{\tau \subseteq \sigma} \Sigma^*_{\sigma}(\tau).\]

A fan refined by $\B(n)$ is smooth if and only if each maximal cone is labeled by a tree poset~\cite[Corollar 3.10]{PRW}.
In the smooth case, containment of cones can be found by contracting edges in Hasse diagrams~\cite[Proposition 3.5 (2)]{PRW}.
When we contract an edge in a Hasse diagram we merge the equivalence classes labeling to the two vertices of the edge.
For a smooth fan each cone will be labeled by a tree preposet.
Figure~\ref{fig:tree} shows the Hasse diagram of a poset indexing a smooth cone along with all its contractions.
The poset on the left indexes the $3$-dimensional cone $\sigma$ defined by $x_4 \leq x_2$, $x_2 \leq x_1$, and $x_2 \leq x_3$.
The other preposets in Figure~\ref{fig:tree} index cones arising from intersecting $\sigma$ with one or more hyperplanes from $\{H_{12}, H_{23}, H_{24}\}$.
We record the two facts from~\cite{PRW} aforementioned in this paragraph in forms we will make make use of them.

\begin{lemma}
A fan refined by $\B(n)$ is smooth if and only if each maximal cone is labeled by a tree poset.
\label{lem:smooth}
\end{lemma}

\begin{lemma}
The faces of any smooth cone $\sigma$ in a fan refined by $\B(n)$ are all cones labeled by contractions of the Hasse diagram of the preposet labeling $\sigma$.
\label{lem:contract}
\end{lemma}

\begin{figure}
    \centering
    \begin{tikzpicture}[scale=0.9]
    \node (topl) at (0,3) {$\{1\}$};
    \node[right of=topl] (topr) {$\{3\}$};
     \coordinate (cent) at ($(topl)!0.5!(topr)$);
     \node[below of=cent] (mid) {$\{2\}$};
      \node[below of=mid] (bot) {$\{4\}$};
      \draw[thick] (topl) -- (mid);
      \draw[thick] (topr) -- (mid);
      \draw[thick] (mid) -- (bot);
      
    \node (topl) at (2.5,3) {$\{1\}$};
    \node[right of=topl] (topr) {$\{3\}$};
    \coordinate (cent) at ($(topl)!0.5!(topr)$);
     \node[below of=cent] (bot) {$\{2,4\}$};
     \draw[thick] (topl) -- (bot) -- (topr);
      
      \node (top) at (5,3) {$\{1\}$};
      \node[below of=top] (mid) {$\{2,3\}$};
       \node[below of=mid] (bot) {$\{4\}$};
       \draw[thick] (top) -- (mid) -- (bot);
       
             \node (top) at (7,3) {$\{3\}$};
      \node[below of=top] (mid) {$\{1,2\}$};
       \node[below of=mid] (bot) {$\{4\}$};
       \draw[thick] (top) -- (mid) -- (bot);
       
       \node (top) at (9,3) {$\{1\}$};
       \node[below of=top] (bot) {$\{2,3,4\}$};
       \draw[thick] (top) -- (bot);
       
       \node (top) at (11,3) {$\{3\}$};
       \node[below of=top] (bot) {$\{1,2,4\}$};
       \draw[thick] (top) -- (bot);
       
       \node (top) at (13,3) {$\{1,2,3\}$};
       \node[below of=top] (bot) {$\{4\}$};
       \draw[thick] (top) -- (bot);
       
       \node (top) at (15,3) {$\{1,2,3,4\}$};
    \end{tikzpicture}
    \caption{The Hasse diagram of a tree poset and all its contractions.}
    \label{fig:tree}
\end{figure}
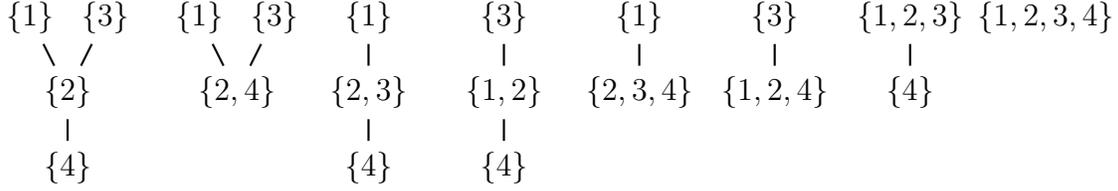

A \emph{linear order} is a poset in which any two elements are comparable.
Under our cone to poset correspondence, linear orders index the Weyl chambers.
Given a poset $P$ and $a \in P$, the \emph{down-degree} of $a$ is the number of edges $b \to a$ in the Hasse diagram while the \emph{up-degree} is the number of edges $a \to b$.
An \emph{upset} of a poset $P$ is a subset $A$ of $[n]$, the underlying set of $P$, such that if $a \in A$ and $b \in [n]$, then $b \in A$ whenever $a \preceq b$.
A \emph{downset} is defined analogously.
Note any element $a \in [n]$ generates an upset (downset) consisting of all elements greater than $a$ (less than $a$).
For any preposet $P$ we let $P^{op} = \{(j,i) : (i,j) \in P\}$.

\section{Smooth fans and strong factorization}
\label{sec:main}

Before proving our main theorem, we provide a lemma to describe the ray generators of a smooth cone labeled by a tree preposet.
It will be useful to us to have the ray generators explicitly described in terms of the Hasse diagram of a tree preposet.
We first define some notation.
Given a tree preposet $P$, for each edge $a \to b$ in its Hasse diagram we will associate an certain element of the lattice $N$ that we now describe.
If we remove the edge $a \to b$ from the Hasse diagram we obtain two connected components.
Let $B$ denote the union of all the equivalence classes of corresponding to the vertices in the component with $b$.
Then to the edge $a \to b$ we associate the lattice vector $v_{a \to b}$ which is defined as $v_{a \to b} := e_B$.

\begin{lemma}
Let $\sigma$ be a smooth cone labeled by a tree preposet $P$ with Hasse diagram $D$.
The ray generators of $\sigma$ are $\{v_{a \to b} : a \to b \text{ is an edge in } D\}$.
\label{lem:ray}
\end{lemma}
\begin{proof}
Let us consider a maximal dimensional cone $\sigma$.
Cones of smaller dimension can be treated by applying the same argument projected onto a smaller dimension space.
Each edge $a \to b$ of the Hasse diagram gives a facet of the cone $\sigma$ on the hyperplane $H_{ab}$.
Since our cone is smooth it has $n-1$ facets and $n-1$ ray generators.
Each ray generator is obtained by the intersection of $n-2$ facets.
Choose any edge $a \to b$ of the Hasse diagram.
We will now show that $v_{a \to b}$ is the ray generator obtained by taking the intersection of the facets corresponding to the other $n-2$ edges.
Let $A$ denote the union of all vertices in the component of $a$ in the Hasse diagram with $a \to b$ removed.
Similarly, let $B$ denote the union of the vertices in the component with $b$ after removing $a \to b$.
The intersection of the $n-2$ hyperplanes corresponding to the other edges is the line defined by $x_i = x_j$ for $i,j \in A$ and $x_k = x_{\ell}$ for $k, \ell \in B$.
Considering again the edge $a \to b$ we see within the cone $\sigma$ that on this line $x_i \leq x_k$ for $i \in A$ and $k \in B$.
Therefore it follows that $v_{a \to b} = e_B$ is a ray generator for cone $\sigma$.
\end{proof}

\begin{theorem}
Let $\Sigma$ be a complete smooth fan refined by $\B(n)$, then there exists of sequence of fans
\[(\Sigma_0, \Sigma_1, \dots, \Sigma_{\ell})\]
such that $\Sigma_0 = \Sigma$, $\Sigma_{\ell} = \B(n)$, and $\Sigma_i$ is obtained from $\Sigma_{i-1}$ by a star subdivision for each $1 \leq i \leq \ell$.
\label{thm:main}
\end{theorem}

\begin{proof}
Let $\Sigma = \Sigma_i$ at some $i$ in the proposed sequence of fans.
Note that $\Sigma = \B(n)$ if and only if every maximal cone is label by a linear order.
So, assume some maximal cone is  not labeled by a linear order.
Since $\Sigma$ is smooth by Lemma~\ref{lem:smooth} this maximal cone of is labeled by a tree poset which is not a linear order.
Choose a cone $\sigma \in \Sigma$ labeled by a tree poset $P$ with Hasse diagram $D$ which is not a linear order.
This means in $D$ there is a vertex with either down-degree or up-degree strictly greater than $1$.
Let us assume we have a vertex $b$ with down-degree $k > 1$.
We may assume we are in the case of down-degree greater than $1$ since we could consider $P^{op}$ and the change of coordinates exchanging $x_i$ and $-x_i$ for $1 \leq i \leq n$.

Let $a_i \prec b$ for $1 \leq i \leq k$ be the covering relations in $P$ with $b$ as the greater element.
Now set $B$ to be the upset generated by $b$ and $A_i$ to be the downset generated by $a_i$ for each $1 \leq i \leq k$.
We can then contract the Hasse diagram of $P$ to a Hasse diagram of a tree preposet having vertex set $\{A_1, A_2, \dots, 
A_k, B\}$ and edge set $\{A_i \to B : 1 \leq i \leq k\}$.
By Lemma~\ref{lem:contract} this tree preposet indexes a face $\tau \subseteq \sigma$.
We will perform a star subdivision relative to $\tau$.

Let $A = A_1 \uplus A_2 \uplus \cdots \uplus A_k$ and let $\overline{A_i}$ denote the complement of $A_i$ in $[n]$.
By Lemma~\ref{lem:ray} the ray generators of $\tau$ are then $\{v_{A_i \to B} = e_{\overline{A_i}} : 1 \leq i \leq k\}$.
Hence, the star subdivision adds the ray generated by $e_B$ since
\begin{align*}
e_{\overline{A_1}} +  e_{\overline{A_2}}+ \cdots +e_{\overline{A_k}} &= ke_B + (k-1)e_A\\
&=e_B + (k-1)e_{[n]}\\
&=e_B.
\end{align*}

We now verify that $\Sigma^*(\tau)$ is still refined by $\B(n)$.
To do this we let
\[S = \{v_{i \to j} : i \to j \in D, i \to j \neq a_1 \to b\} \uplus \{e_B\}\]
and will show that $\sigma' = \cone(S)$ is the cone indexed by the poset with Hasse diagram $D'$ that has edges
\[\{i \to j \in D : j \neq b\} \uplus \{a_1 \to b\} \uplus \{a_i \to a_1 : 1 < i \leq k\}.\]
The labeling of the $\{a_i : 1 \leq i \leq k\}$ is arbitrary it suffices the consider $i = 1$ as we have in $\sigma'$.
Furthermore, it did not matter which cone $\sigma$ containing the face $\tau$ we originally chose.
A local picture of the Hasse diagrams $D$ and $D'$ can be found in Figure~\ref{fig:proofmain}.

Let $v'_{i \to j}$ denote the ray generators of $\sigma'$ corresponding to edges of $D'$.
We need to show
\[\{v'_{i\to j} : i \to j \in D'\} = S.\]
First note if $i \to j \in D$ and $i \to j \in D'$, then $v_{i \to j} = v'_{i \to j}$.
Next we see that $v_{a_i \to b} = v'_{a_i \to a_1}$ for $1 < i \leq k$.
Finally we observe that $e_B = v'_{a_1 \to b}$.

It follows that $\sigma'$ is indeed the cone indexed by the poset wit Hasse diagram $D'$.
Thus, $\Sigma^*(\tau)$ is refined by $\B(n)$.
Let $\Sigma_{i+1} = \Sigma^*(\tau)$.
The theorem is proven by iterating the process we have described.
\end{proof}

\begin{figure}
\begin{tikzpicture}
\node at (0,0.75) {$\vdots$};
\node (b) at (0,0) {$b$};
\node (a1) at (-1,-1) {$a_1$};
\node (a2) at (-0.5,-1) {$a_2$};
\node at (0.25,-1) {$\cdots$};
\node (ak) at (1,-1) {$a_k$};
\node at (-1,-1.5) {$\vdots$};
\node at (-0.5,-1.5) {$\vdots$};
\node at (1,-1.5) {$\vdots$};
\draw[thick] (b)--(a1);
\draw[thick] (b)--(a2);
\draw[thick] (b)--(ak);

\node at (4,1.25) {$\vdots$};
\node (b) at (4,0.5) {$b$};
\node (a1) at (4,-0.5) {$a_1$};
\node (a2) at (4,-1.5) {$a_2$};
\node at (4.75,-1.5) {$\cdots$};
\node (ak) at (5.5,-1.5) {$a_k$};
\node at (3.55,-0.9) {$\rddots$};
\node at (4,-2) {$\vdots$};
\node at (5.5,-2) {$\vdots$};
\draw[thick] (b)--(a1);
\draw[thick] (a1)--(a2);
\draw[thick] (a1)--(ak);
\end{tikzpicture}
\caption{Portions of the Hasse diagrams $D$ and $D'$ from the proof of Theorem~\ref{thm:main}.}
\label{fig:proofmain}
\end{figure}
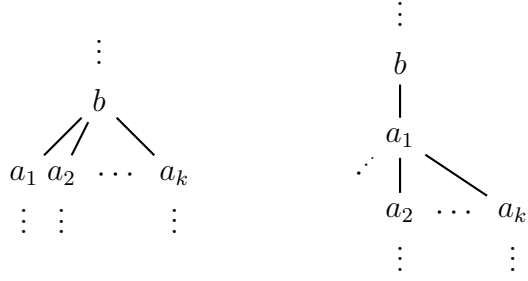

Theorem~\ref{thm:main} gives an affirmative solution to Conjecture~\ref{conj:Oda} in the special case we are considering.

\begin{corollary}
Conjecture~\ref{conj:Oda} holds whenever $\Sigma_1$ and $\Sigma_2$ are two complete smooth fans refined by $\B(n)$. Moreover, in this case the third fan $\Sigma_3$ can always be taken to be $\B(n)$.
\end{corollary}

\end{document}